\documentclass[12pt]{amsart}

\usepackage[T1]{fontenc}
\usepackage{lmodern}
\usepackage{amsmath,amssymb,amsthm}
\usepackage{geometry}
\usepackage{microtype}

\newtheorem{theorem}{Theorem}[section]

\newcommand{\D}{\mathcal D}
\newcommand{\C}{\mathbb C}
\newcommand{\kD}{k_{\D}}
\newcommand{\kapD}{\kappa_{\D}}

\title{Isometries of the Diamond}
\author{Armen Edigarian}
\thanks{Funded by the National Science Centre, Poland under the Weave UNISONO, UMO-2025/07/Y/ST1/00146}
\subjclass[2020]{Primary 32F45; Secondary 32H02.}
\begin{document}

\begin{abstract}
Chavan and Zwonek recently proved that every $C^1$ Kobayashi distance
isometry of the diamond is holomorphic or antiholomorphic; see \cite{CZ}.
We show that the $C^1$ assumption is superfluous.
\end{abstract}

\maketitle

\section{Introduction}

Let
\[
\D:=\{(z_1,z_2)\in\C^2: |z_1|+|z_2|<1\}.
\]
We denote by $\kapD$ the Kobayashi--Royden metric of $\D$ and by $\kD$
its Kobayashi distance.

Chavan and Zwonek proved that every $C^1$ Kobayashi distance isometry
of $\D$ is holomorphic or antiholomorphic; see \cite{CZ}. We obtain the
following equivalent formulation, which in particular removes the
regularity assumption for distance isometries.

\begin{theorem}\label{thm:KR}
Let $F:\D\to\D$ be any map. The following conditions are equivalent:
\begin{enumerate}
\item $F$ is of class $C^1$ and
$\kapD(F(z);dF_zX)=\kapD(z;X)$ for any $z\in\D$ and any $X\in\C^2$;
\item $\kD(F(z),F(w))=\kD(z,w)$ for any $z,w\in\D$;
\item up to a permutation of the coordinates, either
\[
F(z_1,z_2)=(\omega_1z_1,\omega_2z_2)
\quad\text{ or }
F(z_1,z_2)=(\omega_1\overline{z_1},\omega_2\overline{z_2}),
\]
where $|\omega_1|=|\omega_2|=1$.
\end{enumerate}
\end{theorem}

\section{Proof}
It is easy to see that $D$ is convex, so Lempert’s theorem gives 
$\kD=c_D$ and $\kapD=\gamma_D$, where $c_D$ denotes Carath\'eodory distance and $\gamma_D$ denotes the Carath\'eodory–Reiffen metric; see \cite{JP, Lempert}.

The Minkowski functional of $\D$ is $h(z)=|z_1|+|z_2|$.
Since $\D$ is balanced and convex, the radial discs are complex
geodesics. Consequently,
\begin{equation}\label{eq:radial}
\kD(0,z)
=
p(0,h(z))
=
\operatorname{arctanh}(|z_1|+|z_2|),
\end{equation}
where $p$ denotes the Poincar\'e distance on $\mathbb D$.
In particular, $\D$ is complete with respect to $\kD$. On compact subsets of a bounded
hyperbolic domain, the Kobayashi topology agrees with the Euclidean
topology; see \cite[Proposition 2.6.1]{JP}.

\begin{proof}[Proof of Theorem~\ref{thm:KR}]
Let us show that condition {\rm (1)} implies condition
{\rm (2)}. Note that $\D$ is Kobayashi hyperbolic and, therefore,
$\kapD(z;X)>0$. The equality
\[
\kapD(F(z);dF_zX)=\kapD(z;X)
\]
implies that $dF_zX\neq0$ whenever $X\neq0$. The source and target have the same real dimension, so $dF_z$ is an
isomorphism between real tangent spaces. By the inverse
function theorem, $F$ is a local $C^1$ diffeomorphism.

Royden's theorem states that $\kD(z,w)$ is the infimum of
\[
\int \kapD(\gamma(t);\gamma'(t))\,dt
\]
over all piecewise $C^1$ curves $\gamma$ joining $z$ to $w$; see
\cite{Royden,JP}. It follows that every piecewise $C^1$ curve $\gamma$
satisfies $L_\kappa(F\circ\gamma)=L_\kappa(\gamma)$.
The same identity holds for every local inverse of $F$.

Let
$\alpha:[0,1]\to\D$ be a piecewise $C^1$ path with
$\alpha(0)=F(a)$, and let $\widetilde\alpha$ be its local lift through $a$. Suppose that a maximal lift is defined only on $[0,T)$ for some
$T<1$. If $s<t<T$, then Royden's length formula gives
\[
\kD\bigl(\widetilde\alpha(s),\widetilde\alpha(t)\bigr)
\leq
L_\kappa\bigl(\widetilde\alpha|_{[s,t]}\bigr)
=
L_\kappa\bigl(\alpha|_{[s,t]}\bigr).
\]
The last quantity tends to $0$ as $s,t\to T$. Thus
$\widetilde\alpha(t)$ is a $\kD$-Cauchy curve as $t\to T$.
By completeness, there exists $a_T\in\D$ such that
$\widetilde\alpha(t)\to a_T$. By continuity,
\[
F(a_T)
=
\lim_{t\to T}F(\widetilde\alpha(t))
=
\alpha(T).
\]
Since $F$ is a local diffeomorphism at $a_T$, a local inverse near
$\alpha(T)$ extends the lift beyond $T$, contradicting the maximality
of $T$. Hence every such path lifts to the whole interval $[0,1]$.

We have $F$ is a local homeomorphism, so path lifts with a prescribed initial point are unique.  Since $F$ is a local homeomorphism and every path admits a lift with prescribed initial point, the curve-lifting criterion implies that $F$ is a covering map (see 
\cite[Theorem 4.19]{Forster}).
From the fact that $\D$ is connected and simply connected, this covering is trivial, and hence $F$ is a global $C^1$ diffeomorphism.

Finally, Royden's length formula and the infimum over all curves give
\[
\kD(F(z),F(w))\leq\kD(z,w).
\]
Since $F^{-1}$ is $C^1$, the defining infinitesimal equality implies
that $F^{-1}$ also preserves $\kapD$. Applying the same argument to
$F^{-1}$ gives the reverse inequality. Therefore
\[
\kD(F(z),F(w))=\kD(z,w),
\]
as required.

Condition {\rm (3)} clearly implies both {\rm (1)}
and {\rm (2)}. It remains to prove that {\rm (2)} implies {\rm (3)}.

We know that $\kapD=\gamma_{\D}$, hence, for each
$z\in\D$, the function $X\mapsto\kapD(z;X)$ is a norm on the real
tangent space.
Moreover, $\gamma_{\D}$ is locally Lipschitz on $T\D$ (see \cite{JP}, Proposition 2.7.1 (c)). Thus $\kapD=\gamma_{\D}$ is a $C^{0,1}_{\mathrm{loc}}$ Finsler metric in the sense of Matveev--Troyanov.

By Royden's theorem, the associated Finsler distance is precisely
$\kD$; see \cite{Royden,JP}.  Note that $F$ is injective and continuous because it preserves $k_D$. By invariance of domain, $F(\D)$ is open in $\D$, since $\D\subset\mathbb R^4$.
Since $(\D,\kD)$ is complete and $F$ is an isometry, $F(\D)$ is complete, hence closed in $\D$. Thus $F(\D)=\D$ and 
$F$ is bijective. Matveev and Troyanov prove that a
bijective isometry between $C^{k,\alpha}$ Finsler metrics, with $k+\alpha>0$,
is of class $C^{k+1,\alpha}$; see \cite[Corollary~C]{MT}. Taking
$k=0$ and $\alpha=1$, we obtain that $F$ is of class
$C^{1,1}_{\mathrm{loc}}$, and in particular $F$ is $C^1$.

The theorem of Chavan and Zwonek \cite[Theorem~1]{CZ} now applies and
gives condition {\rm (3)}.
\end{proof}

\end{document}